\documentclass[12pt]{article}%
\usepackage{graphicx}
\usepackage{color} 
\usepackage{a4}
\usepackage{amsfonts}
\usepackage{amssymb}
\usepackage{xcolor}
\usepackage{amsmath,amsthm}
\usepackage{amsmath}%
\providecommand{\U}[1]{\protect \rule{.1in}{.1in}}
\newtheorem{theorem}{Theorem}
\newtheorem*{lemmaa}{Lemma A}
\newtheorem*{lemmab}{Lemma B}
\newtheorem*{lemmac}{Lemma C}
\newtheorem*{theorema}{Theorem A}
\newtheorem*{theoremb}{Theorem B}

\newtheorem{corollary}[theorem]{Corollary}

\newtheorem{definition}[theorem]{Definition}

\newtheorem{lemma}[theorem]{Lemma}

\newtheorem{notation}[theorem]{Notation}

\newtheorem{proposition}[theorem]{Proposition}

\newtheorem{terminology}[theorem]{Terminology}

\begin{document} 

\title{On the Complex of Separating meridians in Handlebodies}
\author{Charalampos Charitos, Ioannis Papadoperakis, Georgios Tsapogas \\ 
Agricultural University of Athens}
\maketitle

\begin{abstract}
For a handlebody of genus $g\geq6$ it is shown that every automorphism of the
complex of separating meridians can be extended to an automorphism on the
complex of all meridians and, in consequence, it is geometric. \newline%

\end{abstract}

\section{Definitions and statements of results}
For a compact surface $S,$ the complex of curves $\mathcal{C}\left(  S\right)
,$ introduced by Harvey in \cite{Har}, has vertices the isotopy classes of
essential, non-boundary-parallel simple closed curves in $S.$ A collection of
vertices spans a simplex exactly when any two of them may be represented by
disjoint curves. The complex of curves $\mathcal{C}\left(  S\right)$ as well as
several subcomplexes of $\mathcal{C}\left(  S\right)$ have played an important 
role in the study  of the mapping class group of $S.$  
Ivanov (see for example \cite{Iva}) was the first to prove that every automorphism
of  $\mathcal{C}\left(  S\right)$ is geometric, that is, it is induced by an element of
the mapping class group of $S.$

A prominent subcomplex 
of $\mathcal{C}\left(  S\right)$ is the subcomplex $\mathcal{SC}\left(  S\right)$ 
whose vertices are separating curves. See below, at the end of this section, for a discussion
concerning the subcomplex $\mathcal{SC}\left(  S\right)$ as well as motivation for this work.
The arc complex of $S,$ denoted by $\mathcal{A}\left(  S\right)  ,$ is defined
analogously, with curves replaced by arcs. The arc complex has been studied by
several authors (see \cite{Ha}, \cite{IrMC}, \cite{Iva}).

Similarly, for a $3-$manifold $M,$ the disk complex $\mathcal{M}\left(
M\right)  $ is defined by using the proper isotopy classes of compressing
disks for $M$ as the vertices. It was introduced in \cite{MC}, where it was
used in the study of mapping class groups of $3-$manifolds. In \cite{M-M}, it
was shown to be a quasi-convex subset of $\mathcal{C}\left(  \partial
M\right)  .$ By $H_{g}$ we denote the $3-$dimensional handlebody of genus $g.$
We regard $\mathcal{M}\left(  H_{g}\right)  $ as a subcomplex of
$\mathcal{C}\left(  \partial H_{g}\right)  .$

\begin{definition}
\label{protos} Let $\mathcal{SM}\left(  H_{g}\right)  $ be the simplicial
complex with vertex set being the isotopy classes of separating meridians in
$H_{g}.$ The $k-$simplices are given by collections of $k+1$ vertices having
disjoint representatives.
\end{definition}

Note that the dimension of $\mathcal{SM}\left(  H_{g}\right)  $ is $2g-4.$
Hence, for $g=2$ the complex $\mathcal{SM}\left(  H_{g}\right)  $ is just an
infinite set of vertices so we assume that $g\geq3.$ It is well-known that
$\mathcal{SM}\left(  H_{g}\right)  $ is connected for $g\geq3.$ We will not
need this result in the sequel, however, an easy proof of this result can be
obtained by the general technique presented in \cite[Lemma 2.1]{Put}
considering the action of the mapping class group of $H_{g}$ on $\mathcal{SM}%
\left(  H_{g}\right)  $ and using the specific set of generators for the
mapping class group of $H_{g}$ given in \cite[Corollary 3.4, page 99]{FoMa}.
If $D$ is a separating meridian in $H_{g},$ the vertex containing $D,$ i.e.
the isotopy class containing $D$ will again be denoted by $D.$

The aim of this paper is to show

\begin{theorem}
\label{fi_ext}Every automorphism of $\mathcal{SM}\left(  H_{g}\right)  $
extends uniquely to an automorphism of the complex of all meridians
$\mathcal{M}\left(  H_{g}\right)  ,$ provided that $g\geq6.$
\end{theorem}

It is shown in \cite{KS} that the automorphism group of $\mathcal{M}\left(
H_{g}\right)  $ is isomorphic to the mapping class group of $H_{g}.$ In
particular, every automorphism of $\mathcal{M}\left(  H_{g}\right)  $ is
geometric. Thus, we have

\begin{corollary}
\label{geometric_f}Every automorphism of $\mathcal{SM}\left(  H_{g}\right)  $
is geometric i.e., it is induced by an element of the mapping class group of
$H_{g},$ provided that $g\geq6.$
\end{corollary}

Moreover, we have

\begin{corollary}
\label{karamain} The mapping class group of $H_{g}$ is isomorphic to the
automorphism group of $\mathcal{SM}\left(  H_{g}\right)  ,$ provided that
$g\geq6.$
\end{corollary}

The proof of Theorem \ref{fi_ext} is based on the following basic property
(see Theorem \ref{maindelta}) of genus $1$ separating meridians preserved by
automorphisms of $\mathcal{SM}\left(  H_{g}\right)  $: if two genus
$1-$meridians have common non-separating meridian, so do their images under an
automorphism of $\mathcal{SM}\left(  H_{g}\right)  .$ This property allows the
extension of an automorphism of $\mathcal{SM}\left(  H_{g}\right)  $ to the
whole complex of meridians.

In Section \ref{properties_of_autom} we introduce terminology and prove that
certain intersection properties between separating meridians are preserved by
the automorphisms of $\mathcal{SM}\left(  H_{g}\right)  .$

In Section \ref{arc_auto} we restrict our attention inside a genus $2$
handlebody with $2$ spots on the boundary and we show (see Proposition
\ref{extensionF}) that an automorphism $\phi$ of $\mathcal{SM}\left(
H_{g}\right)  $ induces an automorphism on the arc complex of the boundary
surface of type $(2,2)$ which is known to be geometric (see \cite{IrMC} where
it is shown that the mapping class group of the boundary surface is isomorphic
to the group of automorphisms of the arc complex). Using this, we show the
above mentioned basic property for genus $1-$meridians which live inside a
genus $2$ handlebody with $2$ spots on the boundary. This is the point where 
the assumption $g\geq 6$ is required.

In Section \ref{terminal_extension} we show that the basic property holds in
the case where the two genus $1-$meridians live in the complement of a cut
system, that is, a maximal collection of pairwise disjoint non-separating
meridians which split the handlebody $H_{g}$ to a $3-$ball. Finally, we extend
the property to arbitrary genus $1-$meridians by using the fact that the
complex of cut systems is connected (see \cite{Wa}).

The motivation for this work comes from \cite{[BM]}. In that paper $S$
is a closed surface of genus $g\geq 4.$ The  Torelli subgroup $I(S)$  
of the extended mapping class group $Mod(S)$
of $S$ acts, by definition, trivially on the first homology group 
$H=H_{1}(S,Z)$ of $S.$ Let $K=K(S)$ be the subgroup of $I(S)$ which is
generated by Dehn twists about separating curves. The group $K$ arises
naturally as the kernel of the Johnson map from $I(S)$ to $\wedge ^{3}(H)/H$
(see \cite{[John]}) and it is related to the Heegaard splittings of the
homology 3-sphere (the reader may find additional information in \cite{[BM]}, 
as well as, in \cite{[FM]}). More precisely, in \cite{[BM]} the
commensurator group $Comm(K)$ is considered and, based on ideas of Farb and
Ivanov \cite{[FI]} which have proved a similar result for $I(S),$ the authors
proved that $Comm(K)\cong Aut(K)\cong Mod(S),$ (see Main Theorem 1 in 
\cite{[BM]}) and deduce significant corollaries. The proof of Main Theorem 1 in 
\cite{[BM]}  deploys a reformulation of group theoretic statements in terms
of maps of curve complexes and a basic step is Theorem 1.5 of \cite{[BM]}
which asserts that the automorphism of the complex $SC(S)$ of separating
curves of $S$ is isomorphic to $Mod(S).$ 

On the other hand, in \cite{[Hensel]} the mapping class group $Mod(H_{g})$
of a 3-dimensional handlebody $H_{g}$ of genus $g$ is studied in comparison
with the mapping class group of surfaces. The boundary surface $S=\partial
H_{g}$ is considered and many similarities and differences of $Mod(H_{g})$
are presented with respect to $Mod(S).$ An important topic commented in \cite%
{[Hensel]} is the action of $Mod(H_{g})$ on the first homology group of $%
H_{g}.$ Thus, similarly to the discussion above, the Torelli group $I(H_{g})$
and the group $K(H_{g})$ generated by Dehn twists about separating meridians
can be considered. In this frame, Corollary 4 above can be useful, as in
the case of surfaces, in an effort to relate $K(H_{g})$ with $Aut(K(H_{g}))$
and $Mod(H_{g}).$

\section{Properties of Automorphisms of $\mathcal{SM}\left(  H_{g}\right)  $
\label{properties_of_autom}}

We first give definitions and notation. Throughout this section $\phi$ will
denote an arbitrary automorphism of $\mathcal{SM}\left(  H_{g}\right)  .$ All
intersections between arcs and disks are assumed transverse and minimal. 
Arcs will always be in the boundary surface of the handlebody and disks will 
always be properly embedded. In particular, the boundary of a disk will be considered as 
a simple closed curve in the boundary surface of the handlebody.

By the absolute value $\lvert \cdot \rvert$ of an intersection between arcs and disks  
we will mean the number of components
which, in the case of an intersection between disks, it is a number of arcs and, 
in the case of an intersection between arcs, it is a number of singletons. 
By the intersection between an arc and a disk we will mean the intersection of 
the arc with the boundary curve of the disk.

\begin{definition}
Let $D$ be a separating meridian splitting $H_{g}$ into two handlebodies of
genus, say, $k$ and $g-k$ where $1\leq k\leq g-1.$ Such a separating meridian
will be called a $\left(  k,g-k\right)  -$meridian and the handlebodies will
be called the genus $k$ and $g-k$ components of $D$ and will be denoted by
$H_{k}\left(  D\right)  $ and $H_{g-k}\left(  D\right)  $
respectively.\newline For each $\left(  1,g-1\right)  -$meridian $D,$ the
genus $1$ handlebody $H_{1}\left(  D\right)  $ contains a unique
non-separating meridian disjoint from $D$ which will be denoted by
$\delta \left(  D\right)  .$\newline Let $D_{i}, i=1,2$ be disjoint
separating meridians of type $\left(  k_{i},g-k_{i}\right)  $
and let $\tau$ be an arc properly
embedded in $\partial H_{g}$ with one endpoint in $\partial D_{1}$ and the
other in $\partial D_{2}.$ By the union $D_{1}\cup_{\tau}D_{2}$ of $D_{1}$ and
$D_{2}$ along $\tau$ we mean the separating meridian whose boundary is
obtained by joining $\partial D_{1},\partial D_{2}$ along the arc $\tau$ where
it is implicit that the interior of $\tau$ is disjoint from both $\partial
D_{1},\partial D_{2}.$ We also say that $D_{1}\cup_{\tau}D_{2}$ is the
meridian obtained by joining $D_{1},D_{2}$ along $\tau.$\newline Let $A,B$ be
two disjoint separating meridians. Then $H_{g}\setminus \left(  A\cup B\right)
$ consists of 3 components: a handlebody $H_{k}\left(  A\right)  $ of genus
$k$ with one spot, namely $A,$ a handlebody $H_{k^{\prime}}\left(  B\right)  $
of genus $k^{\prime}$ with one spot, namely $B,$ and a genus $m=g-k-k^{\prime
}$ handlebody bounded by $A$ and $B,$ denoted by $H_{m}\left(  A,B\right)  .$
\end{definition}

We will also use the notion of the sum of two flag complexes. Recall that a
complex $K$ is \textit{a flag complex} if the following property holds: if
$\left \{  v_{0},\ldots,v_{n}\right \}  $ is a set of vertices with the property
the edge $\left(  v_{i},v_{j}\right)  $ exists for all $0\leq i,j\leq n,$ then
$\left \{  v_{0},\ldots,v_{n}\right \}  $ is a simplex. Observe that
$\mathcal{M}\left(  H_{g}\right)  $ as well as $\mathcal{SM}\left(
H_{g}\right)  $ are flag complexes.

If $K,L$ are simplicial complexes we will write $K\oplus L$ to denote the
(flag) complex defined as follows:

\begin{itemize}
\item[(1)] the vertices of $K\oplus L$ is the union of the vertices of $K$ and
the vertices of $L.$

\item[(2)] for every $k-$simplex $\left \{  v_{0},\ldots,v_{k}\right \}  ,$
$k\geq0$ in $K$ and $l-$simplex $\left \{  w_{0},\ldots,w_{l}\right \}  ,$
$l\geq0$ in $L$ there exists a $\left(  k+l+1\right)  -$simplex $\left \{
v_{0},\ldots,v_{k},w_{0},\ldots,w_{l}\right \}  $ in $K\oplus L.$ In other
words, by (1) and (2) the vertices and the edges are defined and then we
require $K\oplus L$ to be the (unique) flag complex generated by these.
\end{itemize}

\begin{definition}
We will say that a complex $M$ splits, equivalently $M$ admits a splitting, if
there exist subcomplexes $K,L$ such that $M=K\oplus L.$
\end{definition}

\begin{lemma}
\label{preserves}Let $\phi:\mathcal{SM}\left(  H_{g}\right)  \rightarrow
\mathcal{SM}\left(  H_{g}\right)  $ be an automorphism. \newline \textbf{(a)}
$D$ is a $\left(  k,g-k\right)  -$meridian in $H_{g},$ $1<k<g-1$ if and only
if the link $Lk\left(  D\right)  $ of $D$ in $\mathcal{SM}\left(
H_{g}\right)  $ splits as $K\oplus L$ where $\dim K=2k-3,$ $\dim L=2\left(
g-k\right)  -3$ and each of $K,L$ does not split.\newline$D$ is a $\left(
1,g-1\right)  -$meridian in $H_{g}$ if and only if $Lk\left(  D\right)  $ does
not split. \newline \textbf{(b)} If $D$ is a $\left(  k,g-k\right)  -$meridian
in $H_{g},0<k<g,$ then $\phi \left(  D\right)  $ is also a $\left(
k,g-k\right)  -$meridian.\newline \textbf{(c)} Let $D$ be a $\left(
k,g-k\right)  -$meridian in $H_{g},$ with $1<k<g-1$ with splitting $Lk\left(
D\right)  =K\oplus L.$ If $E,F\in K$ then $\phi \left(  E\right)  ,\phi \left(
F\right)  $ belong to the same summand of $Lk\left(  \phi \left(  D\right)
\right)  .$ Moreover, $Lk\left(  \phi \left(  D\right)  \right)  =\phi \left(
K\right)  \oplus \phi \left(  L\right)  .$
\end{lemma}

\begin{proof}
The proof of (a) is straightforward by dimension arguments on $Lk\left(
D\right)  .$\\
{\bf (b)} An automorphism $\phi$ maps $Lk\left(  D\right)  $ onto $Lk\left(
\phi \left(  D\right)  \right)  $ isomorphically and, thus, preserves the
splitting (resp. non-splitting).\\
{\bf (c)} Pick a separating meridian $X  $ intersecting both
$E$ and $F$ but disjoint from $D.$  
$\phi \left(  D\right)  $ is separating, hence, $Lk\left(  \phi \left(
D\right)  \right)  $ splits, say, $Lk\left(  \phi \left(  D\right)  \right)  =
K^{\prime}\oplus L^{\prime}. $ If $\phi \left(  E\right)  ,\phi \left(
F\right)  $ do not belong to the same summand of the splitting of $Lk\left(
\phi \left(  D\right)  \right)  $ then $\phi \left(  X
\right)  $ would have to intersect $\phi \left(  D\right)  ,$ a contradiction.
\end{proof}

Before proceeding with further properties of an automorphism $\phi$ we need a
result concerning the curve complex of a sphere with holes (see \cite{Kor}):
\begin{multline}
\mathrm{If\ }n\geq5\mathrm{\ then\ all\ elements\ of\ }Aut\left(
\mathcal{C}\left(  S_{0,n}\right)  \right)  \mathrm{\ are\ geometric,}\\
\mathrm{that\ is,\ they\ are\ induced\ by\ a\ homeomorphism\ of\ }%
S_{0,n}.\label{Korkmaz}%
\end{multline}

\noindent A \textit{cut system} $\mathbb{C}$ for the handlebody $H_{g}$ is a
collection $\left \{  C_{1},\ldots,C_{g}\right \}  $ of pairwise disjoint
non-separating meridians such that $H_{g}\setminus \cup_{i=1}^{g}C_{i}$ is a
$3-$ball with $2g$ spots. We will denote this spotted ball by $H_{g}%
\setminus \mathbb{C}.$

A \textit{separating cut system} $\mathbb{Z}$ for the handlebody $H_{g}$ is a
collection $\left \{  Z_{1},\ldots,Z_{g}\right \}  $ of pairwise disjoint
$\left(  1,g-1\right)  $ meridians so that the intersection $\cap_{i=1}%
^{g}H_{g-1}\left(  Z_{i}\right)  $ is a $3-$ball with $g$ spots. We will
denote this spotted ball by $H_{g}\setminus \mathbb{Z}.$ Note that given
$\mathbb{C}$ we can find $\mathbb{Z}$ so that $C_{i}$ is the unique
non-separating meridian in $H_{1}\left(  Z_{i}\right)  .$ By Lemma
\ref{preserves}, $\phi \left(  Z_{i}\right)  ,i=1,\ldots,g$ is also a
separating cut system, that is, $\cap_{i=1}^{g}H_{g-1}\left(  \phi \left(
Z_{i}\right)  \right)  $ is a $3-$ball with $g$ spots.
Thus, to prove Theorem \ref{fi_ext}, it suffices to show that every 
automorphism of $\mathcal{SM}\left(  H_{g}\right)$ fixing  $\mathbb{Z}$
extends uniquely to an automorphism of $\mathcal{M}\left(  H_{g}\right) , $
provided that $g\geq 6 .$

Every simple closed curve in the boundary of the spotted $3-$ball
$H_{g}\setminus \mathbb{Z}$ is a separating meridian and vice-versa. Thus, if
$g\geq5$ an automorphism $\phi$ of $\mathcal{SM}\left(  H_{g}\right)  $
fixing  $\mathbb{Z}$
induces an element in $Aut\left(  \mathcal{C}\left(  S_{0,g}\right)  \right)
$ which is geometric by (\ref{Korkmaz}). Thus we have

\begin{theorema}
Let $\mathbb{Z}=\left \{  Z_{1},\ldots,Z_{g}\right \}  $ be separating cut
system for $H_{g},$ $g\geq5,$ and $\phi$ an automorphism of $\mathcal{SM}%
\left(  H_{g}\right)  $ fixing  $\mathbb{Z}.$ Then $\phi$ acts geometrically on the subcomplex
$\left \{  D\in \mathcal{SM}\left(  H_{g}\right)  :D\cap Z_{i}=\emptyset,\forall
i=1,\ldots,g\right \}  .$
\end{theorema}

From now on,  $\phi$ will always denote an automorphism of 
$\mathcal{SM}\left(  H_{g}\right)  $  fixing  $\mathbb{Z}$ where, after composing by a homeomorphism of $H_{g},$ we may assume
that $\phi$ is the identity on every separating meridian not intersecting
$\cup_{i=1}^{g}Z_{i}.$ The next Lemma contains certain intersection properties between separating meridians 
which are preserved by an automorphism of $\mathcal{SM}\left(  H_{g}\right).$

\begin{lemma}
\label{presrvesNUMintersection} Let $E,F,D$ be three pair-wise disjoint
separating meridians and $E\cup_{\tau}F$ the separating meridian obtained by
joining $E$ and $F$ along an embedded arc $\tau$ which has one endpoint in
$\partial E$ and the other in $\partial F.$ \newline \textbf{(a)} $\phi \left(
E\cup_{\tau}F\right)  $ is a separating meridian obtained by joining
$\phi \left(  E\right)  $ with $\phi \left(  F\right)  $ along a unique (up to
isotopy) simple arc $\tau^{\prime}.$ \newline \textbf{(b)} Assume $g\geq5$ and
$\tau$ intersects $\partial D$ in $1$ point. Then the arc $\tau^{\prime}$
posited in part (a) above intersects $\phi \left(  D\right)  $ in $1$ point.
\newline \textbf{(c)} Assume $g\geq5$ and $\tau$ intersects $\partial D$ in $2$
points. Moreover, assume $D$ is a $\left(  g_{D},g-g_{D}\right)  $-meridian
with $g-g_{D}\geq3$ and $E,F\subset H_{g-g_{D}}\left(  D\right)  .$ Then the
arc $\tau^{\prime}$ posited in part (a) above intersects $\phi \left(
D\right)  $ in $2$ points. \newline \textbf{(d)} Assume $g\geq 6$ and $D$ is a
$\left(  g_{D},g-g_{D}\right)  $-meridian with $g-g_{D}=2$ and $E,F\subset
H_{g-g_{D}}\left(  D\right)  .$ Let $D^{\prime}$ be a $\left(  4,g-4\right)
$-meridian with $D\subset H_{4}\left(  D^{\prime}\right)  $ such that
$D,D^{\prime}$ bound a genus $2$ handlebody $H_{2}\left(  D,D^{\prime}\right)
=H_{g_{D}}\left(  D\right)  \cap H_{4}\left(  D^{\prime}\right)  $ with 2
spots and, in addition, $D^{\prime}\cap \tau=\varnothing.$ If $\tau$ intersects
$\partial D$ in $2$ points then the arc $\tau^{\prime}$ posited in part (a)
above intersects $\phi \left(  D\right)  $ in $2$ points.
\end{lemma}

\begin{proof}
\textbf{(a)} Let $E$ be a $\left(  g_{E},g-g_{E}\right)  $-meridian and $F$ a $\left(
g_{F},g-g_{F}\right)  $-meridian. Without loss of generality, we may assume
that $H_{g_{E}}\left(  E\right)  \cap H_{g_{F}}\left(  F\right)  =\emptyset.$
Then $X=E\cup_{\tau}F$ is a $\left(  g_{E}+g_{F},g-g_{E}-g_{F}\right)  $
meridian and $X,E,F$ are disjoint meridians which bound a $3-$ball in
$H_{g_{E}+g_{F}}\left(  X\right)  .$ By Lemma \ref{preserves}, all these properties hold for their
images $\phi \left(  E\right)  ,\phi \left(  F\right)  $ and $\phi \left(
X\right)  ,$ hence, there is a unique (up to homotopy) arc $\tau^{\prime}$
with endpoints in $\phi \left(  E\right)  ,\phi \left(  F\right)  $ not
intersecting $\phi \left(  X\right)  .$ As $\phi \left(  E\right)  ,\phi \left(
F\right)  $ and $\phi \left(  X\right)  $ bound a $3-$ball containing
$\tau^{\prime}$ we clearly have $\phi \left(  X\right)  =\phi \left(  E\right)
\cup_{\tau^{\prime}}\phi \left(  F\right)  .$

\textbf{(b)} Observe that if $D$ is a $\left(  1,g-1\right)  -$meridian the
intersection $\partial D\cap \tau$ cannot be one point, hence, $D$ is
necessarily a $\left(  g_{D},g-g_{D}\right)  $-meridian with $1<g_{D}<g-1.$
The meridian $X=E\cup_{\tau}F$ intersects $D$ in one arc which splits both
$X,D$ into two subdisks. By surgery along this arc we obtain four separating
meridians $W_{i},$ $i=1,\ldots,4$ which bound a $3-$ball containing $X$ and
$D.$ Clearly, we may find a separating cut system 
$\mathbb{Z} = \left \{  Z_{1},\ldots,Z_{g}\right \}  $
such that for each $i=1,2,3,4$ $W_{i} \in H_{g}\setminus \mathbb{Z}.$ In particular,
$X$ and $D$ belong to the subcomplex 
$\left \{ C \in \mathcal{SM}\left(  H_{g}\right)  :C\cap Z_{i}=\emptyset,\forall
i=1,\ldots,g\right \}  ,$ hence, by Theorem A, $\phi$ acts geometrically on 
$X$ and $D.$ Therefore, $\phi (D)$ and $\phi \left(  X\right)  =\phi \left(  E\right) \cup_{\tau^{\prime}}\phi \left(  F\right)  $ intersect at a single arc.
This implies that $\tau^{\prime}$  intersects $\phi \left(  D\right)  $ in $1$ point.

\textbf{(c)} We may assume that $E$ is a $\left(  g_{E},g-g_{E}\right)  $-meridian with
$H_{g-g_{E}}\left(  E\right)  \supset D,F$ and similarly for $F,$ that is,
$H_{g-g_{F}}\left(  F\right)  \supset D,E.$ Observe that for an arbitrary arc
$\sigma$ joining $\partial E$ with $\partial F$ we have, by Lemma
\ref{preserves}(c)
\begin{multline}
\phi \left(  E\right)  ,\phi \left(  F\right)
\mathrm{\ belong\ to\ the\ same\ summand\ of\ }Lk\left(  \phi \left(  D\right)
\right) \\
\Leftrightarrow \left \vert \sigma \cap D\right \vert \mathrm{\ is\ even}
\label{sigma1}%
\end{multline}
and, equivalently,
\begin{multline}
\phi \left(  E\right)  ,\phi \left(  F\right)
\mathrm{\ do\ not\ belong\ to\ the\ same\ summand\ of\ }Lk\left(  \phi \left(
D\right)  \right) \\
\Leftrightarrow \left \vert \sigma \cap D\right \vert \mathrm{\ is\ odd}
\label{sigma2}%
\end{multline}
We first show the result in the case $D,E,F$ do not bound a $3-$ball which is
equivalent to $g_{E}+g_{D}+g_{F}<g.$ \begin{figure}[ptb]
\begin{center}
\includegraphics
[scale=1.6]
{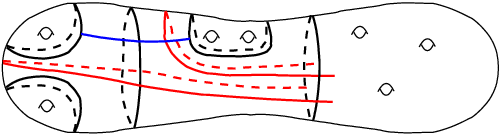}
\end{center}
\par
\begin{picture}(22,12)
\put(51,13){$\phi(F)$}
\put(51,130){$\phi(E_1)$}
\put(97,15){$\phi \left( Y \left(\tau_1 \right) \right)$}
\put(233,20){$\phi \left( Y \left(\tau \right)  \right)$}
\put(203,116){$\phi(E_2)$}
\put(90,103){$\color{blue} \rho^{\prime}$}
\put(267,72){$\color{red} \partial _1 \phi (D)$}
\put(268,51){$\color{red} \partial \phi (D)$}
\end{picture}
\caption{The component $\partial_{1}\phi(D)$ which intersects $\phi \left(  Y
\left( \tau \right)  \right) $ but not $\phi \left(  Y \left( \tau_{1} \right)
\right) $ }%
\label{quardangle}%
\end{figure}The arc $\tau$ splits into $3$ subarcs $\tau_{E},\tau_{D},\tau
_{F}$ such that $\tau_{E}\cup \tau_{D}\cup \tau_{F}=\tau,$ $\tau_{E}$ (resp.
$\tau_{F}$) has endpoints on $E$ (resp. $F$) and $D$ and $\tau_{D}\subset
H_{g_{D}}\left(  D\right)  .$ Set $Y_{E}=E\cup_{\tau_{E}}D$ and $Y_{F}%
=D\cup_{\tau_{F}}F.$ We claim that $Y_{E}$ is not isotopic to $F.$ For, if
$Y_{E},F$ are isotopic then $Y_{E}$ would be a $\left(  g_{E}+g_{D}%
,g_{F}\right)  $-separating meridian which implies that $g_{E}+g_{D}+g_{F}=g$,
a contradiction. Thus $Y_{E}\neq F$ and, similarly, $Y_{F}\neq E.$

Clearly $D,E$ belong to distinct components of $H_{g}\setminus Y_{F}$ and
since $\tau_{E}\cap \tau_{F}=\emptyset$ we have that
\begin{itemize}
\item[(i)] $\left \vert Y_{F}\cap Y_{E}\right \vert =1.$
\item[(ii)] surgery along $Y_{E}\cap Y_{F}$ produces $E,F,D$ and a separating
meridian of type $\left(  g-g_{E}-g_{D}-g_{F},g_{E}+g_{D}+g_{F}\right)  .$
\item[(iii)] in particular, $D$ along with a subdisk of $Y_{E}$ and a subdisk
of $Y_{F}$ bound a $3-$ball with 2 spots.
\item[(iv)] $\left \vert \tau \cap Y_{E}\right \vert =1=\left \vert \tau \cap
Y_{F}\right \vert .$
\end{itemize}
\noindent Let $\tau^{\prime}$ be the arc given by part (a), that is,
$\phi \left(  E\cup_{\tau}F\right)  =\phi \left(  E\right)  \cup_{\tau^{\prime}%
}\phi \left(  F\right)  .$ Then, by case (b), property (\ref{sigma2}) above and
Lemma \ref{preserves}, all four above properties (denoted by (i)$^{\prime
},\ldots$,(iv)$^{\prime}$) hold for $\tau^{\prime}$ and the images
$\phi \left(  D\right)  $, $\phi \left(  E\right)  $, $\phi \left(  Y_{E}\right)
$, $\phi \left(  F\right)  $ and $\phi \left(  Y_{F}\right)  .$ Clearly,
$\left \vert \tau^{\prime}\cap \phi \left(  D\right)  \right \vert \geq2k$ for a
positive integer $k.$ If $k\neq1$, then there would exist $k-1$ subarcs of
$\tau^{\prime}$ with endpoints on $\phi \left(  D\right)  $ contained in
$H_{g-g_{D}} \left(  \phi \left(  D\right)  \right)  .$ By (iv)$^{\prime},$
these subarcs must be disjoint from $\phi \left(  Y_{E}\right)  $ and
$\phi \left(  Y_{F}\right)  .$ Hence, all these subarcs would have to be
contained in the $3$-ball given by (iii)$^{\prime}$. Since the boundary of
this 3-ball is an annulus, we may perform an elementary isotopy to eliminate
them, showing that $k=1.$

We now examine the case $D,E,F$ bound a $3-$ball or, equivalently,
$g_{E}+g_{D}+g_{F}=g.$ We work under the assumption $g-g_{D}\geq3,$ hence, at
least one of $g_{E},g_{F}$ is $\geq2.$ Without loss of generality we assume
that $g_{E}\geq2.$ Let $E_{1}$ be a $\left(  1,g-1\right)  $-meridian in
$H_{g_{E}}\left(  E\right)  $ and $E_{2}$ a $\left(  g_{E}-1,g-g_{E}+1\right)
$-meridian in $H_{g_{E}}\left(  E\right)  $ so that $E=E_{1}\cup_{\rho}E_{2}$
for some simple arc $\rho$ with endpoints on $\partial E_{1},\partial E_{2}$
and $\rho \cap E=\emptyset.$ Extend $\tau$ to a simple arc $\tau_{1}$ with
endpoints on $\partial E_{1},\partial F.$ Clearly $\left \vert \tau_{1}\cap
D\right \vert =2$ and $E_{1},D,F$ do not bound a $3-$ball. Thus, by the
previous case, $\left \vert \tau_{1}^{\prime}\cap \phi \left(  D\right)
\right \vert =2.$ Set $Y\left(  \tau \right)  =E\cup_{\tau} F $ and $Y\left(
\tau_{1}\right)  =E_{1}\cup_{\tau_{1}}F.$ 
\begin{figure}[ptb]
\begin{center}
\includegraphics
[scale=1.4]
{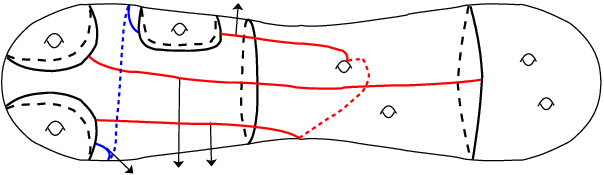} 
\end{center}
\par
\begin{picture}(22,12)
\put(48,20){$\phi(F)$}
\put(51,143){$\phi(B)$}
\put(90,143){$\phi(E)$}
\put(90,15){$\color{blue} \tau^{\prime}$}
\put(163,27){$\phi \left( \Delta \right)$}
\put(301,20){$\phi \left( B^{\prime} \right)$}
\put(143,18){$\color{red} \rho_2^{\prime} $}
\put(157,143){$\color{red} \rho_1^{\prime} $}
\put(118,16){$\color{red} \mu_1^{\prime}$}
\put(272,88){$\color{red} \mu^{\prime}$}
\put(225,113){$\color{red} \rho^{\prime} $}
\end{picture}
\caption{The mutually disjoint arcs $\mu_{1}^{\prime} , \rho_{1}^{\prime} ,
\rho_{2}^{\prime}$ all with one endpoint on $\phi \left(  \Delta \right) $ and
the other on $\phi(B), \phi(E) , \phi(F)$ respectively. }%
\label{quardanglee}%
\end{figure}
Assume, on the contrary,
\[
\left \vert \partial \phi \left(  D\right)  \cap \phi \left(  Y\left(  \tau \right)
\right)  \right \vert >2.
\]
Let $H_{g_{E_{2}}} \equiv H_{g_{E_{2}}} \left(  \phi \left(
Y\left(  \tau \right)  \right)  , \phi \left(  Y\left(  \tau_{1}\right)
\right)  \right) $ be the handlebody bounded by the meridians $\phi \left(
Y\left(  \tau \right)  \right)  $ and $\phi \left(  Y\left(  \tau_{1}\right)
\right)  .$ Then the above inequality implies that $\phi(D)$ intersects
$H_{g_{E_{2}}}$ in more than one components. As $\left \vert \tau_{1}^{\prime
}\cap \phi \left(  D\right)  \right \vert =2 ,$ it follows that at least one
component of $\phi(D) \cap H_{g_{E_{2}}} ,$ say $\partial_{1}\phi \left(
D\right)  ,$ does not intersect $\phi \left(  Y\left(  \tau_{1}\right)
\right)  .$ This implies that $\partial_{1}\phi \left(  D\right)  $ is
contained in the $3$-ball bounded by $\phi \left(  Y\left(  \tau \right)
\right)  $, $\phi \left(  E_{2}\right)  $, $\phi \left(  Y\left(  \tau
_{1}\right)  \right)  $ (see Figure 1). It follows that $\partial_{1}%
\phi \left(  D\right)  $ intersects $\rho^{\prime}$ and, hence, $\partial
\phi \left(  D\right)  $ intersects $\phi \left(  E\right)  =\phi \left(
E_{1}\right)  \cup_{\rho^{\prime}}\phi \left(  E_{2}\right)  $ a contradiction
because $E\cap D=\emptyset.$ This completes the proof of (c).

\textbf{(d)} Since $g\geq6$ we have that the genus of $H_{g-4}\left(  D^{\prime
}\right)  $ is $\geq2,$ hence, we may choose disjoint separating meridians
$B,B^{\prime}$ in $H_{g-4}\left(  D^{\prime}\right)  $ of type $\left(
1,g-1\right)  $ and $\left(  g_{B^{\prime}},g-g_{B^{\prime}}\right)  ,$
respectively, with $1+g_{B^{\prime}}=g-4.$ Moreover, there exists an arc $\mu$
disjoint from $D^{\prime}$ such that $B\cup_{\mu}B^{\prime}=D^{\prime}.$

There is a unique (up to homotopy with endpoints on $\partial E,\partial F$)
arc $\rho$ disjoint from $\tau$ and $D.$ Clearly, $E\cup_{\rho}F=D.$
Pick an arc $\sigma$ with endpoints on $\tau$ and $\partial B$
such that its interior is disjoint from $\tau, D, B$ and $B^{\prime}. $
Clearly, up to isotopy,  $\sigma$   intersects $\partial D^{\prime}$ at 
a single point and it is disjoint from $\mu .$
Set $\Delta$ to be the meridian
\[
\Delta:=\left(  E\cup_{\tau}F\right)  \cup_{\sigma}B.
\]
Clearly, $\Delta$ is a $\left(  g-3,3\right)  $-meridian and $\left \vert
\rho \cap \partial \Delta \right \vert =2.$ By part (c)
\begin{equation}
\phi \left(  D\right)  =\phi \left(  E\cup_{\rho}F\right)  =\phi \left(
E\right)  \cup_{\rho^{\prime}}\phi \left(  F\right)  ,\mathrm{\ with\ }%
\left \vert \rho^{\prime}\cap \phi \left(  \Delta \right)  \right \vert
=2.\label{mer1}%
\end{equation}
Similarly, by part (b),
\begin{equation}
\phi \left(  D^{\prime}\right)  =\phi \left(  B\cup_{\mu}B^{\prime}\right)
=\phi \left(  B\right)  \cup_{\mu^{\prime}}\phi \left(  B^{\prime}\right)
,\mathrm{\ with\ }\left \vert \mu^{\prime}\cap \phi \left(  \Delta \right)
\right \vert =1.\label{mer2}%
\end{equation}
Moreover, $\mu^{\prime}\cap \tau^{\prime}=\varnothing.$

The meridians $\phi \left(  \Delta \right)  ,\phi \left(  E\right)  ,\phi \left(
F\right)  $ and $\phi \left(  B\right)  $ bound a $3$-ball and by (\ref{mer2})
$\mu^{\prime}$ induces an arc $\mu_{1}^{\prime}$ with endpoints on
$\phi \left(  \Delta \right)  ,\phi \left(  B\right)  .$
Similarly, by (\ref{mer1}), $\rho^{\prime}$ induces arcs
$\rho_{1}^{\prime},\rho_{2}^{\prime}$ with endpoints on $\phi \left(
\Delta \right)  ,\phi \left(  E\right)  $ and $\phi \left(  \Delta \right)
,\phi \left(  F\right)  $ respectively, which, up to homotopy, are disjoint
from $\mu_{1}^{\prime}$ (see Figure \ref{quardanglee}).
Every arc with endpoints on $\phi \left(  E\right)  $ and $\phi \left(  F\right)$
which intersects either $\rho_{1}^{\prime}$ or $\rho_{2}^{\prime}$ it must intersect
$\mu_{1}^{\prime}. $   As  $\mu_{1}^{\prime}\cap \tau^{\prime}=\varnothing,$ it follows that
\[
\tau^{\prime}\cap \left(  \rho_{1}^{\prime}\cup \rho_{2}^{\prime}\right)
=\tau^{\prime}\cap \rho^{\prime}=\varnothing
\]
Hence, $\tau^{\prime}$ intersects $\phi \left(  E\right)  \cup_{\rho^{\prime}%
}\phi \left(  F\right)  =\phi \left(  D\right)  $ at $2$ points. This completes
the proof of the Lemma.
\end{proof}

\section{The Arc Complex automorphism\label{arc_auto}}

Assume from now on that $A,B$ are two disjoint separating
meridians which bound a genus $2-$handlebody $H_{2}\left(  A,B\right)  $ with
two spots $A,B.$  We also assume that $H_{k}\left(  A\right)  $ has genus
$k\geq2$ and $H_{k^{\prime}}\left(  B\right)  $ has also genus $k^{\prime}%
\geq2.$ These assumptions on $k,k^{\prime}$ force the assumption $g\geq6$ on the genus of the handlebody.
We will use Lemma
\ref{presrvesNUMintersection} to define an automorphism
on the arc complex of the boundary of $H_{2}\left(  A,B\right)  .$

For arbitrary $A,B$ as above and any automorphism $\phi$ of $\mathcal{SM}%
\left(  H_{g}\right)  ,$ we may assume, using Theorem A, that $\phi$ fixes $A$
and $B$ because there exists a separating cut system disjoint from 
$A\cup B. $ The rest of this section is devoted into showing the following

\begin{proposition}
\label{mainlemmaH2}\textbf{(a)} Let $X,Y\subset H_{2}\left(  A,B\right)  $ be
$\left(  1,g-1\right)  $-separating meridians and $\phi$ an automorphism of
$\mathcal{SM}\left(  H_{g}\right)  ,$ $g\geq6.$ If $\delta \left(  X\right)
=\delta \left(  Y\right)  $ then $\delta \left(  \phi \left(  X\right)  \right)
=\delta \left(  \phi \left(  Y\right)  \right)  .$\newline \textbf{(b)} In
particular, if $A^{\prime},B^{\prime}$ are disjoint separating meridians
bounding a genus $1-$handlebody $H_{1}\left(  A^{\prime},B^{\prime}\right)
,$  the same result as in (a) holds for any $\left(  1,g-1\right)
$-separating meridians $X,Y\subset H_{1}\left(  A^{\prime},B^{\prime}\right)
.$\newline \textbf{(c)} Let $Z$ be a $\left(  2,g-2\right)  $-separating
meridian, $X,Y\subset H_{2}\left(  Z\right)  $ be $\left(  1,g-1\right)
$-separating meridians and $\phi$ an automorphism of $\mathcal{SM}\left(
H_{g}\right)  $ fixing $Z,$ $g\geq6.$ If $\delta \left(  X\right)  =\delta \left(
Y\right)  $ then $\delta \left(  \phi \left(  X\right)  \right)  =\delta \left(
\phi \left(  Y\right)  \right)  .$
\end{proposition}

If $H_{k}\left(  A_{1},\ldots A_{m}\right)  $ is a genus $k$ handlebody with
$m$ spots $A_{1},\ldots A_{m}$ we write $\partial \left(  H_{k}\left(
A_{1},\ldots A_{m}\right)  \right)  $ to denote $\partial H_{k}$ with the
interior $\mathrm{Int}A_{1}\cup \cdots \cup \mathrm{Int}A_{k}$ of the spots
removed. In particular, $\partial H_{2}\left(  A,B\right)  $ is a genus $2$
surface $\Sigma_{2,2}$ with two boundary components $\partial A$ and $\partial
B.$ Denote by $\mathcal{A}$ the arc complex of the boundary surface
$\Sigma_{2,2}$ of $H_{2}\left(  A,B\right)  .$

\begin{proposition}
\label{extensionF}$\phi$ induces an automorphism $\overline{\phi}%
:\mathcal{A}\rightarrow \mathcal{A}$ which is geometric on $\Sigma_{2,2}.$
\end{proposition}

Before we proceed with the proof we need to state the following three Lemmata
A, B and C.

\begin{lemmaa}
Let $\Sigma_{2,2}=\partial H_{2}\left(  A,B\right)  $ be the genus $2$ surface
with $2$ boundary components $\partial A$ and $\partial B.$ \newline%
\textbf{(a)} Let $\sigma,\sigma^{\prime}$ be non-separating simple arcs with
endpoints on $B$ which are not homotopic $\mathrm{rel\,}\partial B.$Then there
exists a simple arc $\tau$ with endpoints on $\partial A$\ and $\partial B$
such that $\tau \cap \sigma=\emptyset$ and $\tau \cap \sigma^{\prime}\neq
\emptyset.$\newline \textbf{(b)} Let $\sigma,\sigma^{\prime}$ be separating
simple arcs with endpoints on $B$ which are not homotopic $\mathrm{rel\,}%
\partial B.$ There exists a simple arc $\tau$ with endpoints on $\partial
A$\ and $\partial B$ such that $\tau \cap \sigma=\emptyset$ and $\tau \cap
\sigma^{\prime}\neq \emptyset.$
\end{lemmaa}

\begin{proof}
(a) Assume $\sigma,\sigma^{\prime}$ have minimal intersection and cut
$\Sigma_{2,2}$ along $\sigma$ to obtain a surface $\Sigma_{1,3}$ with 3
boundary components, namely, $\partial A,$ $\sigma \cup \partial^{+}B$ and
$\sigma \cup \partial^{-}B$ where $\partial^{+}B\cup \partial^{-}B=\partial B.$
$\sigma^{\prime}$ induces arcs $\sigma_{1}^{\prime},\ldots,\sigma_{r}^{\prime
}$ on $\Sigma_{1,3}$ with endpoints on $\sigma \cup \partial
^{+}B$ and/or $\sigma \cup \partial^{-}B.$ We may find an arc $\tau$ in 
 $\Sigma_{1,3}$ with one endpoint on $\partial A$ and the other on
$\partial^{+}B\cup \partial^{-}B$ such that $\tau \cap \sigma_{1}^{\prime}%
\neq \emptyset.$ This arc $\tau$ can be viewed as an arc in 
$\Sigma_{2,2}$ and has the desired property.

The proof in case (b) is straightforward if $\sigma,\sigma^{\prime}$ are
disjoint. If $\sigma \cap \sigma^{\prime}\neq \emptyset,$ let $\Sigma^{+}$ be the
component of $\partial H_{2}\left(  A,B\right)  \setminus \sigma$ containing
$\partial A.$ The boundary $\partial \Sigma^{+}$ is the disjoint union
$\partial A\,$\rotatebox[origin=c]{90}{$\models$}$\, \left(  \sigma \cup
B^{+}\right)  $ where $B^{+}$ is a subarc of $\partial B.$ Then, we may find
an arc $\tau$ in $\Sigma^{+}$ with endpoints in $A\ $and $B^{+}$ disjoint from
$\sigma$ which intersects $\sigma^{\prime}.$
\end{proof}

\begin{notation}
For a simple arc $\tau$ with endpoints on $\partial A$\ and $\partial B$ we
denote by $Z\left(  \tau \right)  :=A\cup_{\tau}B$ the $\left(  2,g-2\right)  $
separating meridian inside $H_{2}\left(  A,B\right)  $ obtained by joining
$\partial A,\partial B$ along the arc $\tau.$\newline
\end{notation}

\begin{lemmab}
Let $\tau$ be simple arc with endpoints on $\partial A$ and $\partial B,$ and
$\sigma$ a simple arc with endpoints on $\partial B.$ Let $E_{1},E_{2}$ be two
disjoint separating meridians inside the component $H_{k^{\prime}}\left(
B\right)  $ of $B$ (not containing $A$) such that $E_2$ and $B$ belong
to the same component of $H_g \setminus E_1 .$ Let $\sigma_{1},\sigma_{2}$ be two
disjoint arcs joining the endpoints of $\sigma$ with $\partial E_{1}$ and
$\partial E_{2}$ respectively, such that 
the interiors of
$\sigma_{1},\sigma_{2}$ are disjoint from $\partial E_1 \cup \partial E_2 .$
Then the following holds%
\[
\sigma \cap \tau=\emptyset \Leftrightarrow \left \vert \phi \left(  E\left(
\sigma \right)  \right)  \cap \phi \left(  Z\left(  \tau \right)  \right)
\right \vert =2.
\]
where $\left \vert \, \, \cdot \, \, \right \vert $ denotes the number of
connected components and $E\left(  \sigma \right)  $ the meridian obtained by
$E_{1}\cup_{\sigma_{1}\cup \sigma \cup \sigma_{2}}E_{2}.$
\end{lemmab}

\begin{proof} Observe that the assumption on $E_2$ and $B$ (belonging
to the same component of $H_g \setminus E_1 $) implies that the interiors of
$\sigma_{1},\sigma_{2}$ are disjoint from $\partial E_1 \cup \partial E_2$ and, hence, $E_{\sigma}$ 
can be defined by $ E_{1}\cup_{\sigma_{1}\cup \sigma \cup \sigma_{2}}E_{2} .$
\newline
Clearly we have
\[
\sigma \cap \tau=\emptyset \Leftrightarrow \left \vert E\left(  \sigma \right)  \cap
Z\left(  \tau \right)  \right \vert =2
\]
and the proof follows immediately from Lemma \ref{presrvesNUMintersection}%
(c,d) applied to the meridians $Z\left(  \tau \right)  ,$ $E_{1}$, $E_{2}$ and
the arc $\sigma$ which intersects $Z\left(  \tau \right)  $ at 2 points.
\end{proof}

By the same manner as in Lemma \ref{presrvesNUMintersection} we obtain 
the following: 

\begin{lemmac}
Let $\sigma,\sigma^{\prime}$ be two simple arcs $\sigma,\sigma^{\prime}$ with
endpoints on $\partial B.$ Let $E_{1},E_{2}$ be two disjoint separating
meridians inside the component $H_{k^{\prime}}\left(  B\right)  $ of $B$ (not
containing $A$) and $\sigma_{1},\sigma_{2}$ (resp. $\sigma_{1}^{\prime}%
,\sigma_{2}^{\prime}$) two disjoint arcs joining the endpoints of $\sigma$
(resp. $\sigma^{\prime}$) with $\partial E_{1}$ and $\partial E_{2}$
respectively. Assume the arcs $\sigma_{1},\sigma_{2},\sigma_{1}^{\prime
},\sigma_{2}^{\prime}$ satisfy $\left(  \sigma_{1}\cup \sigma_{2}\right)
\cap \left(  \sigma_{1}^{\prime}\cup \sigma_{2}^{\prime}\right)  =\emptyset.$
Then the following holds%
\[
\sigma \cap \sigma^{\prime}=\emptyset \Leftrightarrow \left \vert \phi \left(
E\left(  \sigma \right)  \right)  \cap \phi \left(  E\left(  \sigma^{\prime
}\right)  \right)  \right \vert =2.
\]

\end{lemmac}

\begin{proof} $[$of Proposition \ref{extensionF}$]$
\newline \textbf{CASE 1: }We will first define
$\phi \left(  \tau \right)  $ for $\tau$ being a simple arc properly embedded in
$\Sigma_{2,2}=\partial \left(  H_{2}\left(  A,B\right)  \right)  $ with one
endpoint in $\partial A$ and the other in $\partial B.$ For any such arc
$\tau,$
\[
Z\left(  \tau \right)  :=A\cup_{\tau}B
\]
is a $\left(  2,g-2\right)  $-separating meridian inside $H_{2}\left(
A,B\right)  $ obtained by joining $\partial A,\partial B$ along the arc
$\tau.$ Clearly, the opposite is also true: if $X$ is a $\left(  2,g-2\right)
$-separating meridian inside $H_{2}\left(  A,B\right)  $ then $X\ $splits
$H_{2}\left(  A,B\right)  $ into two components. The boundary of the genus $0$
component is a pair of pants whose boundary consists of $\partial A,\partial
B$ and $\partial X.$ There is unique homotopy class of arcs from $\partial A$
to $\partial B$ (disjoint from $\partial X$) and for any arc $\tau$ in this
class we have $\partial A\cup_{\tau}\partial B=\partial X$ and, thus,
$X=A\cup_{\tau}B.$ Since $\phi$ is an automorphism of $\mathcal{SM}\left(
H_{g}\right)  $ we have that $\phi \left(  Z\left(  \tau \right)  \right)
\subset$ $H_{2}\left(  A,B\right)  $ and by Lemma \ref{preserves}(a) it is a
$\left(  2,g-2\right)  $ meridian. By the previous argument there is a unique,
up to homotopy, arc $\overline{\tau}$ from $\partial A$ to $\partial B$
disjoint from $\partial Z\left(  \tau \right)  $ so that $A\cup_{\overline
{\tau}}B=\phi \left(  Z\left(  \tau \right)  \right)  .$ Define $\overline{\phi
}\left(  \tau \right)  :=\overline{\tau}.$

Since the correspondence $\tau \rightarrow Z\left(  \tau \right)  $ is 1-1 the
same holds for  $\overline{\phi}$  viewed as a map on the set of homotopy
classes of arcs from $\partial A$ to $\partial B.$ In other words, if $\tau
_{1},\tau_{2}$ are arcs each having one endpoint in $\partial A$ and the other
in $\partial B,$ then%
\begin{equation}
\tau_{1}\neq \tau_{2}\Leftrightarrow Z\left(  \tau_{1}\right)  \neq Z\left(
\tau_{2}\right)  \Leftrightarrow \overline{\phi}\left(  \tau_{1}\right)
\neq \overline{\phi}\left(  \tau_{2}\right)  \label{arcs1pros1}%
\end{equation}
where $\neq$ means non-homotopic. Similarly, since every $\left(
2,g-2\right)  $-separating meridians inside $H_{2}\left(  A,B\right)  $ has a
pre-image under $\phi,$
\begin{equation}
\overline{\phi}\mathrm{\ is\ onto}.\mathrm{\ }%
\label{arcsonto}%
\end{equation}
viewed as a map on the set of homotopy classes of arcs from
$\partial A$  to $\partial B.$
Moreover,
\[
\tau_{1}\cap \tau_{2}=\emptyset \Leftrightarrow \left \vert Z\left(  \tau
_{1}\right)  \cap Z\left(  \tau_{2}\right)  \right \vert =2.
\]
As $A,B$ and $Z\left(  \tau_{1}\right)  $ are pairwise disjoint and $\tau_{2}$
intersects $Z\left(  \tau_{1}\right)  $ at 2 points, by Lemma
\ref{presrvesNUMintersection}(c) we have that $\phi \left(  Z\left(  \tau
_{1}\right)  \right)  $ and $\phi \left(  Z\left(  \tau_{2}\right)  \right)  $
intersect at 2 points and, thus, $\overline{\phi}\left(  \tau_{1}\right)
\cap \overline{\phi}\left(  \tau_{2}\right)  =\emptyset.$ Working similarly for
the converse we obtain
\[
\tau_{1}\cap \tau_{2}=\emptyset \Leftrightarrow \overline{\phi}\left(  \tau
_{1}\right)  \cap \overline{\phi}\left(  \tau_{2}\right)  =\emptyset.
\]
\underline{Observation 1}\textbf{:} it is straightforward to extend the above
property for any finite collection of pairwise disjoint simple arcs $\tau
_{1},\ldots,\tau_{m}$ from $\partial A$ to $\partial B.$ \\[3mm]%
\underline{Observation 2:} For any simple arc $\rho$ with both endpoints in
$\partial B$ define $n_{\rho}$ to be the maximal number of pairwise disjoint
and non-homotopic simple arcs from $\partial A$ to $\partial B$ with each
being disjoint from $\rho$ (for example, if $\rho$ non-separating, then
$n_{\rho}=4).$ If $\rho,\rho^{\prime}$ are two arcs with endpoints in
$\partial B$ with $\rho$ non-separating and $\rho^{\prime}$ separating, then
clearly $n_{\rho}>n_{\rho^{\prime}}.$\\[2mm]\textbf{CASE 2:} In this case we
will define $\overline{\phi}\left(  \sigma \right)  $ when $\sigma$ is a simple
arc properly embedded in $\Sigma_{2,2}=\partial \left(  H_{2}\left(
A,B\right)  \right)  $ with both endpoints in $\partial B$ and which does not
separate $\Sigma_{2,2}.$ For this, we will use the whole collection of
(isotopy classes of) simple arcs $\tau$ from $\partial A$ to $\partial B$ with
the property $\tau \cap \sigma=\emptyset.$\\[2mm]We claim that there exists a
unique non separating arc $\overline{\sigma}$ with endpoints in $\partial B$
satisfying the following property%
\begin{equation}
\mathrm{for\ every\ arc\ }\tau \mathrm{\ from\ }\partial A\mathrm{\ to\ }%
\partial B\mathrm{\ with\ }\tau \cap \sigma=\emptyset \Rightarrow \overline{\phi
}\left(  \tau \right)  \cap \overline{\sigma}=\emptyset.\label{simpleproof}%
\end{equation}
For the existence of such $\overline{\sigma},$ pick any two disjoint
separating meridians $E_{1},E_{2}$ inside the component $H_{k^{\prime}}\left(
B\right)  $ of $B$ (not containing $A$). Extend $\sigma$ to an arc
$\sigma_{12}=\sigma_{1}\cup \sigma \cup \sigma_{2}$ where $\sigma_{1},\sigma_{2}$
are two disjoint arcs joining the endpoints of $\sigma$ with $\partial E_{1}$
and $\partial E_{2}$ respectively and let $E\left(  \sigma \right)  =E_{1}%
\cup_{\sigma_{12}}E_{2}.$ By Lemma \ref{presrvesNUMintersection}(c,d), the
image $\phi \left(  E\left(  \sigma \right)  \right)  $ must be of the form
$E\left(  \overline{\sigma}\right)  $ for some arc $\overline{\sigma}$ which
intersects $\partial B$ at two points. To check property (\ref{simpleproof}),
let $\tau$ be an arc from $\partial A$ to $\partial B$ with the property
$\tau \cap \sigma=\emptyset.$ Then%
\[%
\begin{array}
[c]{ccl}%
\tau \cap \sigma=\emptyset & \Leftrightarrow & \left \vert \phi \left(  E\left(
\sigma \right)  \right)  \cap \phi \left(  Z\left(  \tau \right)  \right)
\right \vert =2\\
& \Leftrightarrow & \left \vert E\left(  \overline{\sigma}\right)  \cap
Z\left(  \overline{\tau}\right)  \right \vert =2\\
& \Leftrightarrow & \overline{\sigma}\cap \overline{\phi}\left(  \tau \right)
=\emptyset
\end{array}
\]
where the first equivalence is by Lemma B, the second is just an
interpretation of $\phi \left(  E\left(  \sigma \right)  \right)  $ and
$\phi \left(  Z\left(  \tau \right)  \right)  $ and the third is
straightforward.

To see that such an arc $\overline{\sigma}$ is unique, assume $\overline
{\sigma}^{\prime}$ is an other such. Then by Lemma A(a) there exists an arc
$\alpha$ with endpoints on $A\ $and $B$ such that $\alpha \cap \overline{\sigma
}=\emptyset$ and $\alpha \cap \overline{\sigma}^{\prime}\neq \emptyset.$ Since
$\overline{\phi}$ is onto (see property \ref{arcsonto}), there exists an arc
$\tau$ with $\overline{\phi}\left(  \tau \right)  =\alpha.$ Then,
\[
\overline{\phi}\left(  \tau \right)  \cap \overline{\sigma}^{\prime}=\alpha
\cap \overline{\sigma}^{\prime}\neq \emptyset
\]
which means that $\overline{\sigma}^{\prime}$ does not satisfy property (\ref{simpleproof}),
hence, $\overline{\sigma}$ is unique with respect to property
(\ref{simpleproof}). Note also that by Observations 1 and 2, $\overline
{\sigma}$ must be non-separating.

We may now define $\overline{\phi}\left(  \sigma \right)  :=\overline{\sigma}$
where $\overline{\sigma}$ is the above described non-separating arc uniquely
determined by $\sigma.$  
\\
[3mm]\textbf{CASE 3:} In this last case we will
define $\overline{\phi}\left(  \sigma \right)  $ when $\sigma$ is an arc
properly embedded in $\Sigma_{2,2}=\partial \left(  H_{2}\left(  A,B\right)
\right)  $ with both endpoints in $\partial B$ and which separates
$\Sigma_{2,2}.$ \\[2mm]As in Case 2, we will show that there exists a unique
\emph{separating} arc $\overline{\sigma}$ with endpoints on $\partial B$
satisfying the property%
\begin{equation}
\mathrm{for\ every\ arc\ }\tau \mathrm{\ from\ }\partial A\mathrm{\ to\ }%
\partial B\mathrm{\ with\ }\tau \cap \sigma=\emptyset \Rightarrow \overline{\phi
}\left(  \tau \right)  \cap \overline{\sigma}=\emptyset.\label{simpleproofff}%
\end{equation}
Existence and uniqueness of such an arc $\overline{\sigma}$ follows
exactly as in Case 2 by using Lemma A(b) instead of A(a). By Observations 1
and 2, $\overline{\sigma}$ must be separating. We may now define
$\overline{\phi}\left(  \sigma \right)  :=\overline{\sigma}$ where
$\overline{\sigma}$ is the above described separating arc uniquely determined
by $\sigma.$

In an identical way as in Cases 2 and 3, the image of an arc with endpoints in
$\partial A$ is defined. In order to show that $\overline{\phi}$ is a well
defined automorphism of the arc complex $\mathcal{A}$ of $\Sigma_{2,2},$ it
remains to check that disjoint arcs $\sigma,\sigma^{\prime}$ are mapped to
disjoint arcs $\overline{\sigma},\overline{\sigma}^{\prime}$. This is
straightforward if $\partial \sigma \subset \partial A$ and $\partial
\sigma^{\prime}\subset \partial B.$ If $\partial \sigma \subset \partial A$ and
$\sigma^{\prime}$ has one endpoint in $\partial A$ and one in $\partial B,$
the desired property follows from Lemma B. The last case $\partial
\sigma \subset \partial B$ and $\partial \sigma^{\prime}\subset \partial B$
follows from Lemma C since $\sigma,\sigma^{\prime}$ can always be extended to
arcs $\sigma_{12}=\sigma_{1}\cup \sigma \cup \sigma_{2}$ and $\sigma_{12}%
^{\prime}=\sigma_{1}^{\prime}\cup \sigma \cup \sigma_{2}^{\prime}$ so that
$\sigma_{12}\cap \sigma_{12}^{\prime}=\emptyset.$

To check that $\overline{\phi}$ is injective, recall property (\ref{arcs1pros1}) and
observe that  $\overline{\phi},$ by its definition, respects separating (resp.
non-separating) arcs $\sigma$ with $\partial \sigma \subset \partial A$ (or
$\partial \sigma \subset \partial B$). For, if $\sigma,\sigma^{\prime}$ are
separating (resp. non-separating) arcs with  $\overline{\sigma}=\overline
{\sigma}^{\prime},$ by Lemma A(b) (resp. A(a)), there exists an arc $\tau$
from $\partial A$ to $\partial B$ such that
\[
\tau \cap \sigma=\emptyset \mathrm{\ and\ }\tau \cap \sigma^{\prime}\neq \emptyset.
\]
By Lemma B,
\[
\left \vert \phi \left(  E\left(  \sigma \right)  \right)  \cap \phi \left(
Z\left(  \tau \right)  \right)  \right \vert =2\mathrm{\ and\ }\left \vert
\phi \left(  E\left(  \sigma^{\prime}\right)  \right)  \cap \phi \left(  Z\left(
\tau \right)  \right)  \right \vert >2
\]
that is,
\[
\overline{\tau}\cap \overline{\sigma}=\emptyset \mathrm{\ and\ }\overline{\tau
}\cap \overline{\sigma}^{\prime}\neq \emptyset
\]
which contradicts the assumption $\overline{\sigma}=\overline{\sigma}^{\prime
}.$

It is shown in \cite{IrMC} that every (injective) automorphism of the arc
complex of a surface $S_{g,b}$ is geometric provided that $(g,b)\neq
(0,1),(0,2),(0,3), (1,1).$ Thus, the above defined $\overline{\phi
}:\mathcal{A}\rightarrow \mathcal{A}$ is geometric. This completes the proof of
Proposition \ref{extensionF}.
\end{proof}

Clearly, $\overline{\phi}:\mathcal{A}\rightarrow \mathcal{A}$ induces an
automorphism on the curve complex $\mathcal{C}\left(  \Sigma_{2,2}\right)  $
denoted again by $\overline{\phi}.$ We next show that $\overline{\phi}$ agrees
with $\phi$ on the meridian curves in $\Sigma_{2,2}.$

\begin{lemma}\label{ffpaula}
If $X$ is a separating meridian in $H_{2}\left(  A,B\right)  $ then
$\partial \left(  \phi \left(  X\right)  \right)  =\overline{\phi}\left(
\partial X\right)  .$
\end{lemma}

\begin{proof}
If $X$ is a $\left(  2,g-2\right)  $ meridian then, as explained in the proof
of Proposition \ref{extensionF}, Case 1, $X$ is of the form $Z\left(  \tau \right)
=A\cup_{\tau}B$ for some (unique) arc $\tau$ with endpoints on $A\ $and $B$
and the result follows from the definition of $\overline{\phi}.$ If $X$ is a
$\left(  1,g-1\right)  $ meridian, we may choose a $\left(  1,g-1\right)  $
meridian $Z$ disjoint from $X,$ mutually disjoint non-parallel arcs $\tau_{1},\tau_{2}$ with endpoints on $\partial A$ and $\partial B$ not intersecting $\partial Z,\partial X$ and an arc $\tau_{3}$ with endpoints on $\partial A $ and $\partial B$ disjoint from $\partial X,\tau_{1}%
,\tau_{2}$ but with $\tau_{3}\cap \partial Z\neq \emptyset.$ Then, $\Sigma
_{2,2}\setminus \left(  \overline{\tau_{1}}\cup \overline{\tau_{2}}\right)  $
has two components, say $\Sigma^{+},\Sigma^{-},$ each of type $\left(
1,1\right)  .$ By construction, $\overline{\phi}\left(  \partial X\right)  $
must lie in one of the components, say $\Sigma^{+},$ and $\overline{\phi
}\left(  \partial Z\right)  $ as well as $\overline{\tau_{3}}$ must be in
$\Sigma^{-}.$ As there is only one simple separating curve in a surface of
type $\left(  1,1\right)  ,$ which is in fact a meridian curve (parallel to
the boundary), $\overline{\phi}\left(  \partial X\right)  $ is meridian curve.
On the other hand, $\phi \left(  X\right)  $ is a meridian disjoint from
$\overline{\tau_{1}},\overline{\tau_{2}},\overline{\tau_{3}},$ hence, its
boundary is also contained in $\Sigma^{+}.$ This completes the proof.
\end{proof}

\begin{proof}
$[$of Proposition \ref{mainlemmaH2}$]$\newline
\textit{Notation.} For a pair of simple closed curves $\sigma$ and $\tau$ in $\partial H_{2}\left(  A,B\right)$ with $|\sigma \cap \tau|=1 ,$
denote by $N(\sigma ,\tau)$ a tubular neighborhood of the union $\sigma \cup \tau .$ We view $N(\sigma ,\tau)$ as 
a genus $1$ subsurface of $\Sigma_{2,2}$ with single boundary $\partial N(\sigma ,\tau) .$
\newline
(a) Let $X$ be a $\left(1,g-1\right)  -$ meridian in $H_{2}\left(  A,B\right)  $ and $D=\delta \left(
X\right)  $ the (unique) non-separating meridian in $H_{1}\left(  X\right)  .$
We may assume, by Theorem A, that $\phi$ fixes $X$ and, by the above
Lemma, $\overline{\phi}$ fixes $\partial X$ and we will show that
$\overline{\phi}\left(  \partial D\right) = \partial D .$
Assume, on the contrary, that $\overline{\phi}\left(  \partial D\right) \neq \partial D .$

Pick a simple closed curve $\alpha \in\Sigma_{2,2} $ intersecting $\overline{\phi}\left(  \partial D\right)$ 
once such that the boundary curve $\beta := \partial N\left( \alpha ,\overline{\phi}\left(  \partial D\right) \right) $
is not a meridian in  $H_{2}\left(  A,B\right). $ By Proposition \ref{extensionF},
$\overline{\phi}$ is a geometric automorphism on the arc complex, thus, it is induced by a homeomorphism, say $\overline{f},$ of 
$H_{2}\left(  A,B\right) .$ It follows that the inverse image 
$\left(\overline{f}\right)^{-1}\hskip-1.5mm \left(  N\left( \alpha ,\overline{\phi}\left(  \partial D\right) \right) \right)$
is a tubular neighborhood 
of the curves 
$\left(\overline{f}\right)^{-1} \hskip-2mm\left( \alpha\right)$ and $\partial D$ whose boundary is 
$\left(\overline{\phi}\right)^{-1}\hskip-2mm (\beta).$

Since $D$ is a (non-separating) meridian the curve  $\left( \overline{\phi}\right)^{-1}\left( \beta \right) $
is homotopically trivial, that is, it bounds a separating meridian, say $Z.$ Then 
$\partial \left(  \phi \left(  Z\right)  \right) $ clearly bounds a disk but $ \overline{\phi}\left( \partial Z\right) = \beta  $ does not
by construction. 
 This contradicts Lemma 
\ref{ffpaula}.
\newline(c) This follows directly from part (a) because, given a
$\left(  2,g-2\right)  $-separating meridian $Z$ we may find disjoint
separating meridians $A,B\subset H_{g-2}\left(  Z\right)  $ such that $A,B$
and $Z$ bound a $3$-ball. Then, $X,Y\subset H_{2}\left(  A,B\right)
.$\newline(b) If $k>3,$ this follows directly from part (a) because we may
find a meridian $A$ such that  
$H_{2}\left(  A,B^{\prime}\right)\supset H_{1}\left(  A^{\prime},B^{\prime}\right)  $ 
and (a) applies to $H_{2}\left(  A,B^{\prime}\right).$ If $k=2$ then we may find a
meridian $B$ such that  
$H_{2}\left(  A^{\prime} ,B\right)\supset H_{1}\left(  A^{\prime},B^{\prime}\right)  $ 
and (a) applies to $H_{2}\left(  A^{\prime}, B\right).$ 
If $k=1$ then $B$ is a $\left(
2,g-2\right)  $-meridian and the result follows from part (c).
\end{proof}

\section{Extension to the Complex of Meridians\label{terminal_extension}}

In order to extend an automorphism $\phi:\mathcal{SM}\left(  H_{g}\right)
\rightarrow \mathcal{SM}\left(  H_{g}\right)  $ to an automorphism
$\mathcal{M}\left(  H_{g}\right)  \rightarrow \mathcal{M}\left(  H_{g}\right)
$ we will need the following generalization of Proposition \ref{mainlemmaH2}:

\begin{theorem}
\label{maindelta}Let $X,Y$ be $\left(  1,g-1\right)  $-separating meridians in
$H_{g}$ and $\phi$ an automorphism of $\mathcal{SM}\left(  H_{g}\right)  $.
Then
\[
\delta \left(  X\right)  =\delta \left(  Y\right)  \Rightarrow \delta \left(
\phi \left(  X\right)  \right)  =\delta \left(  \phi \left(  Y\right)  \right)
.
\]

\end{theorem}

\noindent Our first task is to generalize Proposition \ref{mainlemmaH2} in the
case where the $\left(  1,g-1\right)  $-separating meridians do not 
intersect a cut system in $H_{g}.$ We will use the following

\begin{terminology}\label{terminology}
(a) Let $Z$ be a genus $\left(  1,g-1\right)  -$meridian and $E$ any separating
meridian disjoint from $\delta \left(  Z\right)  $. A component of
$H_{1}\left(  Z\right)  \cap E$ is called a stripe if it can be isotoped to
the boundary of $H_{1}\left(  Z\right)  .$ We say that $E$ intersects $Z$ in
stripes if every component of $H_{1}\left(  Z\right)  \cap E$ is a
stripe. 

(b) Clearly, any two stripes are parallel, that is, their boundary arcs are isotopic $rel(\partial Z) .$ We will be saying that two (parallel) stripes
$S_1 , S_2 \subset H_{1}\left(  Z\right)\cap E$ are \emph{not nested} if the 
corresponding stripes on the boundary of $H_1 (Z)$ are disjoint and \emph{nested} otherwise.

(c) If $E$ intersects $Z$ in a single stripe we define a
stripe elimination $\overline{E}$ of $E$ to be the meridian obtained from $E$
as shown in Figure 3. There are two non-isotopic ways to perform a stripe
elimination. However, for our purposes, this ambiguity will be irrelevant.
Clearly, $\overline{E}$ is disjoint from $Z.$ Note also that stripe elimination does alter the topological type of a separating meridian, that is, 
\begin{equation}
 E \mathrm{\ is\ a\ } (k,g-k)\mathrm{-meridian} \Leftrightarrow
  \overline{E} \mathrm{\ is\ a\ } (k,g-k)\mathrm{-meridian}
\end{equation}

(d) If $E$ intersects $Z$ in several not nested stripes we may perform a stripe elimination on the outermost stripe and this can be done repeatedly for any 
number of stripes. 

(e) If $E$ intersects $Z$ in stripes and several stripes are nested inside 
an outermost stripe $S\subset H_{1}\left(  Z\right)$ we may perform a stripe elimination on $S$ by eliminating simultaneously all stripes nested in $S.$  

(f) Let $E_1 $ be a  separating meridian intersecting $Z$ in stripes such that the stripes in $H_{1}\left(  Z\right)\cap E_1$ are not nested and let $E_2$ 
be an other separating meridian disjoint from $E_2$ possibly with nested stripes. 
If none of the stripes of $E_1$ is nested inside a stripe of $E_2$ a stripe 
elimination $\overline{E}_2$ of $E_2$ can be performed on an outermost 
stripe of $E_2 $ (see part (e) above) and the resulting meridian satisfies 
$ E_1 \cap \overline{E}_2 =\emptyset .$
\end{terminology}

\begin{figure}[ptb]
\begin{center}
\includegraphics[scale=1.3]
{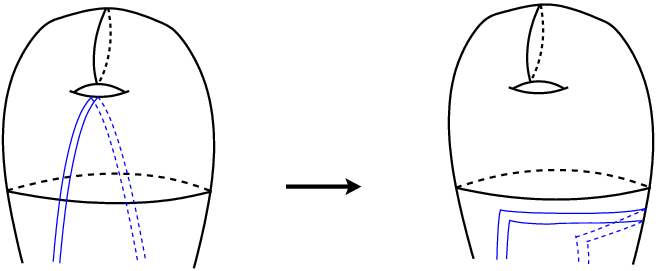}
\end{center}
\par
\begin{picture}(22,12)
\put(42,36){$E$}
\put(323,36){$\overline{E}$}
\put(330,198){$\delta \left(  Z\right)$}
\put(62,195){$\delta \left(  Z\right)$}
\put(139,72){$Z$}
\put(412,74){$Z$}
\end{picture}
\caption{Stripe elimination $\overline{E}$ of $E$ }%
\end{figure}

\begin{lemma}
\label{lemmadelta}Let $\mathbb{C}$ $=\left \{  C_{1},\ldots,C_{g}\right \}  $ be
a cut system in $H_{g},$ $X,Y$ be $\left(  1,g-1\right)  $-separating
meridians both disjoint from $\mathbb{C}$ and $\phi$ an automorphism of
$\mathcal{SM}\left(  H_{g}\right)  $. Then
\begin{equation}
\delta \left(  X\right)  =\delta \left(  Y\right)  \Rightarrow \delta \left(
\phi \left(  X\right)  \right)  =\delta \left(  \phi \left(  Y\right)  \right)  .
\tag{$\ast$}\label{star}%
\end{equation}

\end{lemma}

\begin{proof}
Pick a separating cut system $\mathbb{Z=}\left \{  Z_{1},\ldots,Z_{g}\right \}
$ so that $C_{i}$ is the unique non-separating meridian in $H_{1}\left(
Z_{i}\right)  ,$ that is, $\delta \left(  Z_{i}\right)  =C_{i}.$ 
As the collection  
$\left \{  \phi \left( Z_{1}\right) ,\ldots,\phi \left( Z_{g}\right)\right \}$ 
is again a separating cut system 
we may assume that $\phi$ fixes $\mathbb{Z}$.
\\[4mm]\noindent 
Claim 1: Any
$\left(  1,g-1\right)  $-separating meridian $X$ disjoint from $\mathbb{C}$
has the form $C_{i}\cup_{\sigma}C_{i}$ for some $C_{i}\in \mathbb{C}$ and
$\sigma$ a simple closed curve
in $X$ intersecting $C_{i}$ at a single point. 
\\[3mm]
Proof of Claim 1: Set
$\Delta=\delta(X)$ and let $\sigma$ be the simple closed curve so that
$X=\Delta \cup_{\sigma}\Delta.$
Suppose $\Delta \cap \mathbb{Z}\neq \emptyset$. Then, since $\Delta$ and $X$ are
disjoint from $\mathbb{C},$ we have that $\Delta$ and $\sigma$ lie in the
spotted $3-$ball $H_{g}\setminus \mathbb{C}.$ However, $\Delta$ would be a
separating meridian in $H_{g}\setminus \mathbb{C}$ making $\lvert \partial
\Delta \cap \sigma \rvert=1$ impossible. Thus $\Delta \cap \mathbb{Z}=\emptyset$
which implies that either $\Delta=C_{i}$ for some $C_{i}\in \mathbb{C}$ or,
$\Delta$ lies in the spotted $3-$ball $H_{g}\setminus \mathbb{C}.$ As above,
the latter case is impossible and this completes the proof of Claim
1.
\\[3mm]
Therefore, it suffices to show the result for the meridians $Z_{1}$
and $X$ where $X$ is a $\left(  1,g-1\right)  $-separating meridian in
$H_{g}\setminus \mathbb{C}$ formed by $C_{1}\cup_{\sigma}C_{1}$ where $\sigma$
is a simple closed curve intersecting $C_{1}$ at a single point and
$\sigma \cap C_{i}=\emptyset$ for all $i=2,\ldots,g.$ 
Clearly, such a meridian $X$ intersects $H_1 (Z_i), i=2,\ldots,g$ in parallel stripes which are not nested.
\\[4mm] \noindent 
Claim 2: there exists a $\left(  1,g-1\right)  $-separating meridian $X\left(
g\right)  $ in $H_{g}\setminus \mathbb{C}$ disjoint from $Z_{g}$ such that
(\ref{star}) holds for $X\ $and $X\left(  g\right)  .$ 
\\[3mm]
Proof of Claim 2: Cut along $C_{2},\ldots,C_{g}$ to obtain a genus $1$ handlebody, denoted by
$H\left(  Z_{1}\right)  ,$ with $2\left(  g-1\right)  $ spots denoted by
$C_{2}^{-},C_{2}^{+},C_{3}^{-},C_{3}^{+},\ldots,C_{g}^{-},C_{g}^{+}.$ The
boundary curve $\partial X$ is a separating curve on the boundary of $H\left(
Z_{1}\right)  $ and all $2\left(  g-1\right)  $ spots belong to the same
component. Therefor, we may connect the following sequence of spots%
\[
C_{2}^{-},C_{2}^{+},C_{3}^{-},C_{3}^{+},\ldots,C_{g-1}^{-},C_{g-1}^{+}%
\]
by simple arcs $\tau_{i,i},$ $i=2,\ldots,g-1$ with endpoints on $C_{i}%
^{-},C_{i}^{+}$ and arcs $\tau_{i,i+1},i=2,\ldots,g-2$ with endpoints on
$C_{i}^{+},C_{i+1}^{-}$ so that all these arcs are pairwise disjoint and the
interior of each is disjoint from $\partial X,$ disjoint from $C_{i}^{+}%
,C_{i}^{-}$ for all $i=2,\ldots,g-1$ and disjoint from $C_{g}^{-},C_{g}^{+} .$
We may now successively construct the meridians
\[
E^{2}=C_{2}^{-}\cup_{\tau_{2,2}}C_{2}^{+},E^{3}=E^{2}\cup_{\tau_{2,3}}%
C_{3}^{-}\cup_{\tau_{3,3}}C_{3}^{+},\ldots
\]
to obtain the separating meridian
\[
E^{g-1}=E^{g-2}\cup_{\tau_{g-2,g-1}}C_{g-1}^{-}\cup_{\tau_{g-1,g-1}}%
C_{g-1}^{+}%
\]
which is of type $\left(  2,g-2\right)  .$ By construction, $H_{g-2}\left(
E^{g-1}\right)  $ contains $C_{2},\ldots,C_{g-1}$ whereas $H_{2}\left(
E^{g-1}\right)  $ contains $X$ and $C_{g}.$ Moreover, both $X$
and $E^{g-1}$ intersect $Z_{g}$ in stripes and, as mentioned above, the stripes
in $H_1(Z) \cap X$ are not nested. Again by construction, none of the stripes in $H_1(Z) \cap X$ is nested inside a stripe in $H_1(Z) \cap E^{g-1} .$
\\[2mm]
Case I: $E^{g-1}$ does not intersect $Z_{g}.$ In this case we may perform repeatedly stripe eliminations on $X$ to obtain a
$\left(  1,g-1\right)  -$meridian $\overline{X}$ which does not intersect
$Z_{g}.$ Clearly, $X$ and $X\left(  g\right)  :=\overline{X}$ are contained in
$H_{2}\left(  E^{g-1}\right)  $ so, by Proposition \ref{mainlemmaH2}(c),
(\ref{star}) holds for $X\ $and $X\left(  g\right)  $ as claimed.
\\[2mm]
Case II: $E^{g-1}$ intersects $Z_{g}. $ In this case it suffices to find meridians $\overline
{E}^{g-1},\overline{X}$ so that
\begin{equation}
\left \vert H_{1}\left(  Z_{g}\right)  \cap \overline{X}\right \vert <\left \vert
H_{1}\left(  Z_{g}\right)  \cap X\right \vert ,X\cup \overline{X}\subset
H_{2}\left(  \overline{E}^{g-1}\right)  \mathrm{\ and\ }\delta \left(
X\right)  =\delta \left(  \overline{X}\right)  \label{delta}%
\end{equation}
where $\lvert \cdot \rvert$ denotes number of components. Then, applying this step a finite number of times we reach the desired
meridian $X\left(  g\right)  .$

All components of $H_{1}\left(  Z_{g}\right)  \cap \left(  E^{g-1}\cup
X\right)  $ are (parallel) stripes. If the outermost stripe amongst all
stripes in $H_{1}\left(  Z_{g}\right)  \cap \left(  E^{g-1}\cup X\right)  $ is
a component of $H_{1}\left(  Z_{g}\right)  \cap X$ we may perform a stripe
elimination on $X$ using the outermost stripe which. as mentioned above is
not nested. The resulting meridian
$\overline{X}$ satisfies (\ref{delta}) because
\[
\left \vert H_{1}\left(  Z_{g}\right)  \cap \overline{X}\right \vert
+1\leq \left \vert H_{1}\left(  Z_{g}\right)  \cap X\right \vert .
\]
If the outermost stripe amongst all stripes in $H_{1}\left(  Z_{g}\right)
\cap \left(  E^{g-1}\cup X\right)  $ is a component of $H_{1}\left(
Z_{g}\right)  \cap E^{g-1}$ we may perform (repeatedly, if necessary) stripe
eliminations on $E^{g-1}$ (see Terminology \ref{terminology}(f)) to obtain a meridian $\overline{E}^{g-1}$ such that
the outermost stripe amongst all stripes in $H_{1}\left(  Z_{g}\right)
\cap \left(  \overline{E}^{g-1}\cup X\right)  $ is a component of $H_{1}\left(
Z_{g}\right)  \cap X$ and, hence, the previous case applies. This completes
the proof of Claim 2. 
\\[4mm]
In an identical way we may show that there
exists a $\left(  1,g-1\right)  $-separating meridian $X\left(  g-j\right)  $
in $H_{g}\setminus \mathbb{C}$ disjoint from $Z_{g},\ldots,Z_{g-j}$ for
$j=1,\ldots,g-2.$ such that (\ref{star}) holds for $X\left(  g-j+1\right)
\ $and $X\left(  g-j\right)  .$ At the last step, i.e., $j=g-2,$ we obtain a
meridian $X(2)$ disjoint from $Z_{g},\ldots,Z_{2}$ with
\[
\delta \left(  X\right)  =\delta \left(  X\left(  g\right)  \right)
=\cdots=\delta \left(  X\left(  g-j\right)  \right)  =\cdots=\delta \left(
X\left(  2\right)  \right)  .
\]
If $\left \vert X(2)\cap Z_{1}\right \vert =2$ then $X(2),Z_{1}$ belong to a
genus $1$ handlebody with two boundary spots so we may apply Proposition
\ref{mainlemmaH2}(b) for the meridians $X\left(  2\right)  $ and $Z_{1}$ to
conclude the proof of the Lemma.

Next assume $\left \vert X(2)\cap Z_{1}\right \vert =2k.$ By induction, it
suffices to find a meridian $\overline{X(2)}$ with $\left \vert \overline
{X(2)}\cap Z_{1}\right \vert \leq2k-2$ and $\delta \left(  \overline
{X(2)}\right)  =\delta \left(  X(2)\right)  .$
\newline 
Denote the components of
$X(2)\cap Z_{1}$ by $A_{1},A_{2},\ldots,A_{2k-1},A_{2k}.$ As $X(2)\cap
H_{g-1}(Z_{1})$ consists of $k$ stripes we may denote them by $s_{A_{1}A_{2}}%
,s_{A_{3}A_{4}},\dots,s_{A_{2k-1}A_{2k}}.$ Each stripe $s_{A_{2i-1}A_{2i}}$
splits $\partial Z_{1}$ into two subarcs. We need to know that for some stripe
$s_{A_{2j-1}A_{2j}},$ $j\in \left \{  1,\ldots k\right \}  $ one of the two
subarcs of $\partial Z_{1}$ determined by $s_{A_{2j-1}A_{2j}}$ does not
intersect any of the components $A_{2i-1},A_{2i},$ $i\in \left \{  1,\ldots
,k\right \}  \setminus \left \{  j\right \}  .$ This follows from the fact that
each stripe $s_{A_{2i-1}A_{2i}}$ separates the boundary of the spotted
$3-$ball $H_{g}\setminus \mathbb{Z}$ into $2$ components along with the fact
that the stripes $s_{A_{1}A_{2}},\dots,s_{A_{2k-1}A_{2k}}$ are pairwise
disjoint. Denote by $\tau \left(  Z_{1}\right)  $ the subarc of $\partial
Z_{1}$ not intersecting any of the components $A_{2i-1},A_{2i},$ $i\in \left \{
1,\ldots,k\right \}  \setminus \left \{  j\right \}  .$ Let $s_{0}$ be a stripe
connecting $A_{2j-1}$ with $A_{2j}$ so that both boundaries of $s_{0}$ are
isotopic to $\tau \left(  Z_{1}\right)  .$ We may construct a meridian
$\overline{X(2)}$ by replacing $s_{A_{2j-1}A_{2j}}$ by $s_{0}.$
The components $A_{2j-1},A_{2j}$ can be eliminated by an isotopy so, clearly,
$\overline{X(2)}$ satisfies $\left \vert \overline{X(2)}\cap Z_{1}\right \vert
\leq2k-2.$ Moreover, $\left \vert \overline{X(2)}\cap X(2)\right \vert =2,$
thus, $\overline{X(2)},X(2)$ belong to a genus $1$ handlebody with two spots
and by Proposition \ref{mainlemmaH2}(b), we conclude the proof of the Lemma.
\end{proof}

We will use a result of B. Wajnryb (see \cite{Wa}) which states that the
complex of cut systems is connected. Let $H_{g}$ be a handlebody of genus $g$
with a finite number of spots (that is, disjoint, distinguished disks) on its
boundary. The complex of cut systems is a $2-$dimensional complex with
vertices being cut systems of $H_{g}$ and two cut systems are connected by an
edge if they have $g-1$ meridians in common and the other two are disjoint.
The \textit{cut system complex} $\mathcal{CS}\left(  H_{g}\right)  $ of
$H_{g}$ is defined to be the $2-$dimensional flag complex determined by the
above mentioned vertices and edges, that is, if $\left \{  \mathbb{C}%
_{0},\mathbb{C}_{1},\mathbb{C}_{2}\right \}  $ is a set of vertices with the
property the edge $\left(  \mathbb{C}_{i},\mathbb{C}_{j}\right)  $ exists for
all $0\leq i,j\leq2,$ then $\left \{  \mathbb{C}_{0},\mathbb{C}_{1}%
,\mathbb{C}_{2}\right \}  $ is a $2-$simplex. The following is shown in
\cite{Wa}:

\begin{theoremb}
The cut system complex $\mathcal{CS}\left(  H_{g}\right)  $ is connected and
simply connected.
\end{theoremb}

Since for any given $\left(  1,g-1\right)  $-separating meridian $X$ we may
pick a cut system $\mathbb{C}_{X}$ with $X\subset H_{g}\setminus \mathbb{C}%
_{X},$ the proof of Theorem \ref{maindelta} follows from Lemma
\ref{lemmadelta}, Theorem B and the following:

\begin{lemma}
\label{lemmadeltaadjacent}Let $X,Y$ be $\left(  1,g-1\right)  $-separating
meridians with $\delta \left(  X\right)  =\delta \left(  Y\right)  \equiv \Delta$
and
\[
\mathbb{C}_{X}=\left \{  \Delta,C_{X},C_{3},\ldots,C_{g}\right \}
,\mathbb{C}_{Y}=\left \{  \Delta,C_{Y},C_{3},\ldots,C_{g}\right \}
\]
cut systems connected by an edge in the cut system complex $\mathcal{CS}%
\left(  H_{g}\right)  $ such that $X\subset H_{g}\setminus \mathbb{C}_{X}$ and
$Y\subset H_{g}\setminus \mathbb{C}_{Y}.$ Then for any automorphism $\phi$ of
$\mathcal{SM}\left(  H_{g}\right)  ,$ $\delta \left(  \phi \left(  X\right)
\right)  =\delta \left(  \phi \left(  Y\right)  \right)  .$
\end{lemma}

\begin{proof}
Cutting $H_{g}$ along $\mathbb{C}_{X}\cup \left \{  C_{Y}\right \}  $ we obtain
two components ($3-$balls) with a total of $2g+2$ spots (each meridian gives
rise to $2$ spots).

If the $2$ spots corresponding to $\Delta$ lie on the same component, we may
find a simple closed curve $\sigma \subset \partial H_{g}$ intersecting
$\partial \Delta$ at a single point and, in addition,
\[
\sigma \cap \left(  C_{X}\cup C_{Y}\cup C_{3}\cup \cdots \cup C_{g}\right)
=\emptyset.
\]
Then the meridian $C_{X,Y}=\Delta \cup_{\sigma}\Delta$ clearly has
$\delta \left(  C_{X,Y}\right)  =\Delta$ and satisfies
\[
C_{X,Y}\subset H_{g}\setminus \mathbb{C}_{X}\mathrm{\  \ and\  \ }C_{X,Y}\subset
H_{g}\setminus \mathbb{C}_{Y}.
\]
By Lemma \ref{lemmadelta}, $\delta \left(  \phi \left(  X\right)  \right)
=\delta \left(  \phi \left(  C_{X,Y}\right)  \right)  =\delta \left(  \phi \left(
Y\right)  \right)  .$

Suppose now that the $2$ spots corresponding to $\Delta$ lie on distinct
components of $H_{g}$ $\setminus \left(  \mathbb{C}_{X}\cup \left \{
C_{Y}\right \}  \right)  .$ Denote them by $B^{+}$ and $B^{-}.$ Then $B^{+}$
has spots $\Delta^{+},C_{X}^{+}$ and $C_{Y}^{+}$ corresponding to
$\Delta,C_{X}$ and $C_{Y}$ respectively and similarly for $B^{-}.$ Pick
disjoint arcs $\sigma_{X}^{+}$ and $\sigma_{Y}^{+}$ in $\partial B^{+}$ with
endpoints on $\Delta^{+},C_{X}^{+}$ and $\Delta^{+},C_{Y}^{+}$ respectively.
Pick disjoint arcs $\sigma_{X}^{-}$ and $\sigma_{Y}^{-}$ in $\partial B^{-}$
so that $\sigma_{X}^{-}$ (resp. $\sigma_{Y}^{-}$) has the same endpoints with
$\sigma_{X}^{+}$ (resp. $\sigma_{Y}^{+}$). Then $\sigma_{X}=\sigma_{X}^{+}%
\cup \sigma_{X}^{-}$ and $\sigma_{Y}=\sigma_{Y}^{+}\cup \sigma_{Y}^{-}$ are
simple closed curves in
\[
H_{g}\setminus \left(  C_{Y}\cup C_{3}\cup \cdots \cup C_{g}\right)
,\mathrm{\  \ and\ }H_{g}\setminus \left(  C_{X}\cup C_{3}\cup \cdots \cup
C_{g}\right)
\]
respectively, each intersecting $\Delta$ once. Then, for the meridians
$X_{1}=\Delta \cup_{\sigma_{X}}\Delta$ and $Y_{1}=\Delta \cup_{\sigma_{Y}}%
\Delta$ we clearly have
\[
\delta \left(  X_{1}\right)  =\delta \left(  Y_{1}\right)  =\Delta
\mathrm{\  \ and\  \ }X_{1}\subset H_{g}\setminus \mathbb{C}_{Y},Y_{1}\subset
H_{g}\setminus \mathbb{C}_{X}.
\]
Hence by Lemma \ref{lemmadelta},
\[
\delta \left(  \phi \left(  X\right)  \right)  =\delta \left(  \phi \left(
X_{1}\right)  \right)  \mathrm{\  \ and\  \ }\delta \left(  \phi \left(  Y\right)
\right)  =\delta \left(  \phi \left(  Y_{1}\right)  \right)  .
\]
Moreover, as $\sigma_{X}\cap \sigma_{Y}=\emptyset,$ $\sigma_{Y}$ intersects
$\partial X_{1}$ at $2$ points, thus, there exists a genus $1$ handlebody with
two spots containing both $X_{1}$ and $Y_{1}.$ By Proposition
\ref{mainlemmaH2}(b), $\delta \left(  \phi \left(  X_{1}\right)  \right)
=\delta \left(  \phi \left(  Y_{1}\right)  \right)  $ and, thus, $\delta \left(
\phi \left(  X\right)  \right)  =\delta \left(  \phi \left(  Y\right)  \right)  $
as desired.
This completes the proof of the Lemma and, in turn, of Theorem \ref{maindelta}.
\end{proof}

We may now extend an arbitrary automorphism $\phi:\mathcal{SM}\left(
H_{g}\right)  \rightarrow \mathcal{SM}\left(  H_{g}\right)  $ to an
automorphism $\phi_{M}:\mathcal{M}\left(  H_{g}\right)  \rightarrow
\mathcal{M}\left(  H_{g}\right)  $ on the whole complex of meridians
$\mathcal{M}\left(  H_{g}\right)  .$

\begin{definition}
Let $D\in \mathcal{M}\left(  H_{g}\right)  $ be a non-separating meridian. Pick
any $\left(  1,g-1\right)  $-separating meridian $Z$ with $\delta \left(
Z\right)  =D.$ Define
\[
\phi_{M}:\mathcal{M}\left(  H_{g}\right)  \rightarrow \mathcal{M}\left(
H_{g}\right)
\]
by $\phi_{M}\left(  D\right)  =\delta \left(  \phi \left(  Z\right)  \right)  .$
By Theorem \ref{maindelta}, $\phi_{M}$ is, as a map, well defined.
\end{definition}

Surjectivity of $\phi$ clearly implies surjectivity of $\phi_{M}.$ The inverse
automorphism $\phi^{-1}$ can be extended as above to an automorphism $\left(
\phi^{-1}\right)  _{M}:\mathcal{M}\left(  H_{g}\right)  \rightarrow
\mathcal{M}\left(  H_{g}\right)  $ which satisfies $\left(  \phi^{-1}\right)
_{M}\circ \phi_{M}=\mathrm{Id}_{\mathcal{M}\left(  H_{g}\right)  }.$ Thus, the
map $\phi_{M}$ is injective and surjective. In the sequel we will suppress the
lower index in $\phi_{M}.$

\begin{proposition}
The map $\phi:\mathcal{M}\left(  H_{g}\right)  \rightarrow \mathcal{M}\left(
H_{g}\right)  $ is the unique complex automorphism of $\mathcal{M}\left(
H_{g}\right)  $ extending the given automorphism of $\mathcal{SM}\left(
H_{g}\right)  .$
\end{proposition}
 
\begin{proof}
Let $D,E$ be any two meridians. We must show
\begin{equation}
D\cap E=\emptyset \Leftrightarrow \phi \left(  D\right)  \cap
\phi \left(  E\right)  =\emptyset. \label{twosep}%
\end{equation}
We will only show the "if" direction because, as $\phi$ has an inverse, the converse follows in an identical way. Clearly, if both $D,E$ are separating we have nothing to show. We next examine two cases: 
\\[3mm]
Case I: $D$ is non-separating and $E$ separating. If $E$ is a 
$\left(  1,g-1\right)  $-separating meridian with $\delta (E) =D$ then,
by definition of $\phi :\mathcal{M}\left(  H_{g}\right)  \rightarrow \mathcal{M}\left(H_{g}\right) $ we have $\phi (D) = \delta (\phi (E))$ and, clearly,  
$\phi \left(  D\right)  \cap \phi \left(  E\right)  =\emptyset .$ If either 
$\delta (E) \neq D$ or, $E$ is  a $\left(  k,g-k\right)  $-separating meridian with $2\leq k\leq g-1 $ we may find a $\left(  1,g-1\right)  $-separating 
meridian $X$ with $\delta (X) =D$ and $X\cap E = \emptyset .$ Then 
$\phi \left(  X\right)  \cap \phi \left(  E\right)  =\emptyset $ which implies that $\phi \left(  D\right)  \cap \phi \left(  E\right)  =\emptyset .$ 
\\[3mm]
Case II: $D$ and $E$ are both non-separating. 
We may find two disjoint separating meridians  $A,B$  which bound a genus $2-$handlebody $H_{2}\left(  A,B\right)  $ with
two spots $A,B$ (compare with discussion before Proposition \ref{mainlemmaH2}) such that $D,E \subset H_{2}\left(  A,B\right) . $ Let $X\subset  H_{2}\left(  A,B\right)$ be a $\left(  1,g-1\right)  $-separating meridian with $\delta (X) =D.$ Note that if the union $D\cup E$ separates $H_g$ the intersection $X\cap E$
is necessarily non-empty. In the proof of Proposition \ref{mainlemmaH2}(a), after composing $\phi$ with a geometric automorphism so that $A,B$ and $X$ are fixed, we have shown that $\overline{\phi} (\partial D) =\partial D$ where $\overline{\phi}$ is the induced automorphism of the arc complex. Thus, $\phi$ and $\overline{\phi}$ not only agree on the separating meridians in $H_{2}\left(  A,B\right), $  as shown in Lemma \ref{ffpaula}, but also $\overline{\phi} (\partial D) = \partial \phi(D).$
Since $\overline{\phi}$ is geometric the desired property for 
 $\phi \left(  D\right)  $ and $\phi \left(  E\right)  $ follows.
 
For uniqueness, let $\phi^{\prime}:\mathcal{M}\left(  H_{g}\right)
\rightarrow \mathcal{M}\left(  H_{g}\right)  $ be a complex automorphism such
that $\phi^{\prime}=\phi$ on $\mathcal{SM}\left(  H_{g}\right)  $ and assume
$\phi^{\prime}\left(  D\right)  \neq \phi \left(  D\right)  $ for some
non-separating meridian $D.$ Choose a simple closed curve $\alpha$ with
$\left \vert \alpha \cap \phi \left(  D\right)  \right \vert =1.$ Then, for the
separating meridian $Z=\phi \left(  D\right)  \cup_{\alpha}\phi \left(
D\right)  $ we clearly have $\delta \left(  Z\right)  =\phi \left(  D\right)  $
and $Z\cap \phi^{\prime}\left(  D\right)  \neq \emptyset.$ The latter implies
that the meridians $\left(  \phi^{\prime}\right)  ^{-1}\left(  Z\right)
=\phi^{-1}\left(  Z\right)  $ and $\left(  \phi^{\prime}\right)  ^{-1}\left(
\phi^{\prime}\left(  D\right)  \right)  =D$ intersect. This is a contradiction
because, by definition of $\phi,$ $\delta \left(  \phi^{-1}\left(  Z\right)
\right)  =D.$
\end{proof}

\end{document}